\documentclass[12pt]{amsart}
\usepackage{amssymb,latexsym,comment,url}
\usepackage{mathtools}
\usepackage{adjustbox, enumitem, etoolbox, subscript, tabularx, tabu}
\usepackage{array, float, indentfirst, longtable, tikz, tikz-qtree, xcolor}
   \usetikzlibrary{calc}
\usepackage{caption}
\usepackage[nogroupskip, toc]{glossaries}
\usepackage[dotinlabels]{titletoc}

\newcommand{\ord}{\operatorname{ord}}

\newcommand{\Q}{{\mathbb Q}}

                      {\hspace*{\fill}\nobreak$\Box$\par\medskip}
                       {\hspace*{\fill}\nobreak$\Box$\par\medskip}

\newtheorem{Proposition}{Proposition}[section]

\theoremstyle{definition}

\newtheorem{Remark}[Proposition]{Remark}

\usepackage[utf8]{inputenc}
\usepackage{blindtext}
\usepackage{amssymb}
\usepackage[english]{babel}
\usepackage{amsthm}

\newcommand{\z}{\mathbb{Z}}
\newcommand{\q}{\mathbb{Q}}
\newtheorem{theorem}{Theorem}[section]
\newtheorem{corollary}[theorem]{Corollary}
\newtheorem{lemma}[theorem]{Lemma}
\newtheorem{definition}[theorem]{Definition}
\newtheorem{proposition}[theorem]{Proposition}
\newtheorem{example}[theorem]{Example}
\newtheorem{question}[theorem]{Question}

\begin{document}

\title{Eventual Stability of pure polynomials over the rational field}
\author[M. O. Darwish]%
{Mohamed O. Darwish}
\address{Faculty of Engineering and Natural Sciences, Sabanc{\i} University, Tuzla, \.{I}stanbul, 34956 Turkey}
\email{mdarwish@sabanciuniv.edu}
\author[M. Sadek]%
{Mohammad~Sadek}
\email{mohammad.sadek@sabanciuniv.edu}

\maketitle

\let\thefootnote\relax

\begin{abstract}
A polynomial with rational coefficients is said to be {\em pure} with respect to a rational prime $p$ if its Newton polygon has one slope. In this article, we prove that the number of irreducible factors of the $n$-th iterate of a pure polynomial over the rational field $\q$ is bounded independent of $n$. In other words, we show that pure polynomials are {\em eventually stable}. Consequently, several eventual stability results available in literature follow; including the eventual stability of the polynomial $x^d+c\in\q[x]$, where $c\ne 0,1$, is not a reciprocal of an integer. In addition, we establish the dynamical irreducibility, i.e., the irreducibility of all iterates, of a subfamily of pure polynomials, namely Dumas polynomials with respect to a rational prime $p$ under a mild condition on the degree. This provides iterative techniques to produce irreducible polynomials in $\q[x]$ by composing pure polynomials of different degrees. 
During the course of this work, we 
 characterize all polynomials whose degrees are large enough that are not pure, yet they possess pure iterates. This implies the existence of polynomials in $\z[x]$ whose shifts are all dynamically irreducible.    
\end{abstract}

\section{Introduction}
Let $S$ be a set of polynomials defined over a field $K$. An interesting question is whether one can construct an irreducible polynomial over $K$ using polynomials in $S$. For example, over the rationals, Hilbert's Irreducibility Theorem ensures that there exists infinitely many $c\in \q$ such that $f(x)+c$ is irreducible for any polynomial $f$ with rational coefficients. In fact, if $f$ and $g$ are polynomials in $K[x]$ with $\alpha$ being a root of $f$ in the algebraic closure of $K$, then Capelli's Lemma, \cite[Lemma 0.1]{Capelli}, asserts that
$f\circ g$ is irreducible over $K$ if and only if $f$ is irreducible over $K$ and
$g(x) - \alpha$ is irreducible over $K(\alpha)$.

A modern approach toward the question is to iteratively construct such irreducible polynomials. More precisely, we set $S$ to consist of one irreducible polynomial $f$ in $K[x]$, then we study the irreducibility of the polynomials $f\circ f, f\circ f\circ f,\ldots $. This justifies our interest in the systems introduced in the following definition.
\begin{definition}\cite{Benedetto2019}
  A (discrete) {\em dynamical system} $(A,\phi)$ is a set $A$ together with a self-map $\phi : A \rightarrow A$. The $n$th-iterate of the map $\phi$ is defined by
\begin{equation*}
\phi^n:=\underbrace{\phi \circ \phi\ldots \circ \phi}_\text{$n$-times}    
\end{equation*}
Conventionally, $\phi^0$ is the identity map on $A$.
\end{definition}
In this article, our main interest lies in polynomial maps defined over the rational field $\q$. In particular, we focus on the irreducibility of iterations of a polynomial map. For an irreducible polynomial $f\in \q[x]$, it is not always true that $f^n$ is irreducible for all $n\ge2$. Therefore, we introduce the following definition.
\begin{definition}\label{stable polynomials}\cite{ahmadi_luca_ostafe_shparlinski_2012}
Let $K$ be a field. The polynomial $f$ in $K[x]$ is  called {\em dynamically irreducible}, or stable, if all the iterates $f, f^2,\ldots, f^n,\ldots$ are irreducible over $K$.
\end{definition}
Odoni \cite{Odoni} was the first to establish the concept of dynamical irreducibility (the credit of the term \textit{stable} is attributed to him). In \cite[Lemma 2.2]{Odoni}, it was shown that for a prime ideal $P$ in an integral domain $R$, a $P$-Eisenstein polynomial in $R[x]$ is dynamically irreducible. In addition, he presented the first nontrivial example of a dynamically irreducible polynomial over $\q$, namely the polynomial $x^2-x+1$. The interested reader may consult \cite[Proposition 4.1]{odoni2} for a proof. Stoll \cite{Stoll1992} produced a dynamical irreducibility criterion for quadratic polynomials in $\q[x]$ of the form $f(x)=x^2+a$. This criterion is based on associating the sequence $c_1=-a$ and $c_{n+1}={c_n}^2+a=f^{n+1}(0)$, $n\geq 2$, to the iterations of the quadratic binomial. He proved that if this sequence contains no squares in $\z$, then $f$ is dynamically irreducible, see \cite[Corollary 1.3]{Stoll1992}. In addition, explicit families of dynamically irreducible polynomials of degree $2$ were exhibited in \cite{ahmadi_luca_ostafe_shparlinski_2012} and \cite{Jones}.
Jones \cite{Jones2011} generalized the aforementioned criterion of Stoll over any arbitrary field with Characteristic different from $2$ as follows.
\begin{proposition}\cite[Proposition 2.3]{Jones2011}\label{quadratic stable}
 Let $K$  be a field with characteristic not equal to $2$. Suppose the polynomial $f(x)=ax^2+bx+c \in  K[x]$ has a critical point $\gamma=\frac{-b}{2a}$. Then $f$ is dynamically irreducible over $K$ if $af^2(\gamma),af^3(\gamma),\ldots,af^n(\gamma),\ldots$ and $-af(\gamma)$ are all nonsquares in $K$.
\end{proposition}

In a different direction, Danielson and Fein \cite{Danielson2001} extended Stoll's result for any binomial of the form $x^n-b$ over some specific rings. They were able to deduce the dynamical irreducibility of such a polynomial from the irreducibility of the first iterate, see \cite[Corollary 5]{Danielson2001}. Moreover,  Ali \cite[Corollary 1]{NidalAli} proved that $p^r$-Eisenstein polynomials are dynamically irreducible over $\q$, where $p^r$-Eisenstein polynomials are defined as follows. 
\begin{definition}\cite[Definition 5]{NidalAli}\label{p^r-Eisenstein polynomials}
Let $f(x)=a_dx^d+\ldots+a_0\in \q[x]$. We say $f$ is {\em $p^r$-Eisenstein} if there is a prime $p$ and an integer $r\geq 1$ such that 
\begin{itemize}
    \item[i)] $\nu_p(a_d)=0$
    \item[ii)] $\nu_p(a_i)\geq r$ for all $1 \leq i\leq d-1$
    \item[iii)] $\nu_p(a_0)=r$
    \item[iv)] $\operatorname{gcd}(r,d)=1$
\end{itemize}
\end{definition}
In fact, $p^r$-Eisenstein polynomials constitute a subfamily of $p^r$-Dumas polynomials, see Definition \ref{def: dumas criterion}. A $p^r$-Dumas polynomial is an irreducible polynomial whose Newton polygon with respect to the prime $p$ consists of exactly one line segment. Following a careful analysis of the Newton polygon of such polynomials, we prove the following fact, see Corollary \ref{cor: dumas polynomials are dynamically irreducible}.
\begin{corollary}
Let $f$ and $g$ be $p^r$-Dumas polynomials in $\q[x]$, then $f\circ g$ is $p^r$-Dumas. In particular, a $p^r$-Dumas polynomial is dynamically irreducible over $\q$. 
\end{corollary}
This provides a variety of examples of dynamically irreducible polynomials different from the quadratic and binomial dynamically irreducible polynomials available in literature.  Inspired by Odoni's observation in \cite[Lemma 1.2]{Odoni} that a polynomial with a dynamically irreducible iterate is itself dynamically irreducible, we fully characterize polynomials $f$ that possess a $p^r$-Dumas iterate in the following corollary, see Corollary \ref{cor: Characterization_of_eventually_Dumas} for the proof.
\begin{corollary}
Let $r\ge 1$ be an integer and $p$ be a rational prime. 
Let $f(x)=a_dx^d+\ldots+a_0 \in \q[x]$, $d>r$, be such that $f$ is not $p^r$-Dumas. There is an integer $n\ge 2$ such that $f^n$ is $p^r$-Dumas if and only if the following conditions hold
\begin{itemize}
    \item[i)] $d=p^m$, for some $m\ge 1$,
    \item[ii)] $f(x)\equiv a_dx^{d}+a_0 \pmod p$ with $\nu_p(a_d)=\nu_p(a_0)=0$,
    \item[iii)] $f(x+c)$ is $p^r$-Dumas for some $c\in \q$.
\end{itemize}
\end{corollary}
 One may see easily that if a polynomial $f\in\z[x]$ is dynamically irreducible, it is not necessarily true that all its shifts $f(x+c)$, $c\in\z$, are dynamically irreducible. However, the aforementioned characterization gives rise to polynomials in $\z[x]$ for which all the shifts are dynamically irreducible. 

If $f,\ldots,f^{n-1}$ are irreducible but $f^n$ is reducible, we say that $f$ is {\em $n$-newly reducible}, see \cite{Illig2021} and \cite{katharine}. It is worth mentioning that the irreducibility of the first few iterates of a polynomial does not necessarily imply dynamical irreducibility. For example, if $f(x)=x^2+1\in \mathbb{F}_{43}[x]$, then $f,\ldots,f^5$ are irreducible, but $f^6(x)=g(x)h(x)$ where  
\begin{align*}
g(x)&=x^{32}+13x^{31}+36x^{30}+34x^{29}+7x^{28}+21x^{27}+8x^{26}+11x^{25}+31x^{24}+35x^{23}+11x^{22}\\
&+10x^{21}+9x^{20}+7x^{19}+26x^{18}+35x^{17}+23x^{16}+33x^{15}+4x^{14}+28x^{13}+38x^{12}+17x^{11}\\
&+40x^{10}+39x^9+25x^8+5x^7+42x^6+15x^5+10x^4+25x^3+31x^2+26x+37,\\
h(x)&=x^{32}+30x^{31}+36x^{30}+9x^{29}+7x^{28}+22x^{27}+8x^{26}+32x^{25}+31x^{24}+8x^{23}+11x^{22}\\
&+33x^{21}+9x^{20}+36x^{19}+26x^{18}+8x^{17}+23x^{16}+10x^{15}+4x^{14}+15x^{13}+38x^{12}+26x^{11}\\
&+40x^{10}+4x^9+25x^8+38x^7+42x^6+28x^5+10x^4+18x^3+31x^2+17x+37.    
\end{align*}
In other words, $x^2+1$ is $6$-newly reducible over $\mathbb{F}_{43}$. Even if a polynomial is reducible or newly reducible, one may still construct a tower of irreducible polynomials. For instance, one can find another polynomial $g$ such that $g\circ f^n$ is irreducible for all $n\geq 1$. In the latter case, $g$ is said to be {\em $f$-stable}. In the following corollary, we introduce polynomials $f\in\q[x]$ such that all the iterates of any $p^r$-Dumas polynomials are $f$-stable, see Corollary \ref{cor: composition of dumas and p-type} for a proof.
\begin{corollary}
Let $g$ be a $p^r$-Dumas polynomial of degree $d$ and $f(x)= ax^e+ p^s h(x)\in\q[x]$ be such that $h(x)\in\q[x]$ with $\nu_p(a)=0$, $\deg(h)<e$, $\nu_p(h)\ge 0$, and $s>\frac{r}{d}$. If $\gcd(r,e)=1$, then $g^n\circ f^m$ is irreducible for all $n,m\geq 1$. In particular, $g^n$ is $f$-stable for any $n\geq 1$. 
\end{corollary}
 In addition, we display polynomials $f\in\q[x]$ for which one can find a $p^r$-Dumas polynomial $g\in\q[x]$ with $g\circ f$ being reducible, yet there exists $N\ge 2$ such that $g^n\circ f^m$ is irreducible for all $m\ge 1$ and all $n\ge N$ in the following corollary, for a proof, see Corollary \ref{cor: the iterate of dumas when composed with p-type is dumas}. 
 
 \begin{corollary}
Let $g$ be a $p^r$-Dumas polynomial of degree $d$ and $f(x)= ax^e+ p^s h(x)\in\q[x]$ be such that $h(x)\in\q[x]$ with $\nu_p(a)=0$, $\deg(h)<e$, $\nu_p(h)\ge 0$, and $s>\frac{r}{d}$. If $\gcd(r,e)=1$, then $g^n\circ f^m$ is irreducible for all $n,m\geq 1$. In particular, $g^n$ is $f$-stable for any $n\geq 1$. 
\end{corollary}

In this work, we also give due attention to eventually stable polynomials defined as follows. 
\begin{definition}\cite[Definition 1.1]{Demark}
Let $K$ be a field, $f$ be a polynomial in $ K[x]$, and $\alpha \in K$. We say $(f, \alpha)$ is {\em eventually stable}
over $K$ if there exists a constant $C(f, \alpha)$ such that the number of irreducible factors over $K$ of
$f^n(x)-\alpha$ is at most $C(f, \alpha)$ for all $n \geq 1$. In particular, we say that $f$ is {\em eventually stable} over $K$ if $(f, 0)$ is eventually stable.
\end{definition}
Equivalently, $f$ is eventually stable if there exists an iteration $N\geq 1$ such that the number of irreducible factors does not change in all the succeeding iterations. In fact, finding eventually stable polynomials equips us with an alternative way to construct $g$-stable polynomials for $g\in K[x]$. More precisely, if $f^N=g_1\cdots g_t$ where $g_1,\ldots,g_t$ are irreducible  and the number of irreducible factors of any iterate is at most $t$, then $f^{N+n}=(g_1\circ f^n)\cdots (g_t\circ f^n)$ where $g_1\circ f^n,\ldots,g_t\circ f^n$ are irreducible for all $n\geq 1$. In conclusion, $g_1,\ldots,g_t$ are all $f$-stable. 

In \cite[Corollary 6]{hamblen}, it was proven that binomials of the form $x^d+c \in \q[x]$ are eventually stable over $\q$ whenever $c$ is nonzero and not a reciprocal of an integer. For an overview of eventual stability of quadratic polynomials, we refer the reader to \cite{Demark}. In this work, given a prime 
$p$, we study the dynamical behavior of the iterations of $p^r$-pure polynomials, $r\ge1$, see Definition \ref{def: pure polynomials}. We remark that a $p^r$-Dumas polynomial is an irreducible $p^r$-pure polynomial. Given an upper bound on the number of irreducible factors of iterations of $p^r$-pure polynomials, we establish the eventual stability of $p^r$-pure polynomials in  the following theorem, see Theorem \ref{pure polynomials are eventually stable} for a proof.
\begin{theorem}
Suppose that $f\in\q[x]$ is a $p^r$-pure polynomial of degree $d$. Then for any $n\geq 1$, the iterate $f^n$ has at most $\gcd(d^n,r)$ irreducible factors over $\q$ and each irreducible factor has degree at least $\frac{d^n}{\gcd(d^n,r)}$. Moreover, $f$ is eventually stable over $\q$. 
\end{theorem}

 Consequently, we show that the aforementioned result in \cite[Corollary 6]{hamblen} follows directly from our results.
 In addition, we fully characterize polynomials $f$ that are not $p^r$-pure yet they possess $p^r$-pure iterates, hence they are eventually stable, when $\deg f>r$ in the following theorem, see Theorem \ref{thm: characterization_of_eventually_pure} for a proof.
 \begin{theorem}
Let $r$ be a positive integer and $p$ be a prime. Suppose $f(x)=a_dx^d+\ldots+a_0 \in \q[x]$ is not $p^r$-pure and $d>r$. Then $f(x)$ is eventually $p^r$-pure if and only if the following conditions hold
\begin{itemize}
    \item[i)] $d=p^m$ for some $m\geq 1$,
    \item[ii)] $f(x)\equiv a_dx^d+a_0 \pmod p$ such that $\nu_p(a_d)=\nu_p(a_0)=0$,
    \item[iii)] $f(x+c)$ is $p^r$-pure for some $c\in \q$.
\end{itemize}
Moreover, the least integer $n>1$ such that $f^n$ is $p^r$-pure is given by
$n=p$ if $a_d\equiv 1 \pmod p$; or $n=\operatorname{ord}_p(a_d)$ otherwise.
\end{theorem}
 
 It is worth noting that the eventual stability of the latter family of polynomials may also follow from the criteria given in Theorem \ref{thm: characterization_of_eventually_pure} by observing that these polynomials $f$ have good reduction at $p$ with degree a power of $p$, see \cite[Theorem 1.3]{JonesLevy}. However, we give two examples to show how our results can produce stronger bounds for the number of irreducible factors of the iterates. 
 \begin{example}
Let $k$ be an odd positive integer. Consider the following polynomials:
\begin{align*}
    f(x)&={(x+1)}^{2^m}+2^{6k} \text{ where } 2^m>6k\\
    g(x)&={(x+1)}^{2^m}+2^k  \text{ where } 2^m>k\\
\end{align*}
By \cite[Corollary 4.9]{Jones1150281}, both $f$ and $g$ are eventually stable such that $f^n$ has at most $6k$ irreducible factors while $g^n$ has at most $2^k$. It is easy to see that 
\begin{equation*}
    f(x)=\big({(x+1)}^{2^{m-1}}-2^{\frac{3k+1}{2}}(x+1)^{2^{m-2}}+2^{3k}\big)\big({(x+1)}^{2^{m-1}}+2^{\frac{3k+1}{2}}(x+1)^{2^{m-2}}+2^{3k}\big)
\end{equation*}
Using Theorem \ref{thm: characterization_of_eventually_pure}, $f^n$ possesses at most $\operatorname{gcd}(6k,2^m)=2$ irreducible factors, so any irreducible factor of any iterate of $f$ is $f$-stable. In fact, using Theorem \ref{cor: Characterization_of_eventually_Dumas}, $f$ is dynamically irreducible.
 \end{example}

In \S \ref{sec2}, we study the iterations of $p$-type polynomials and polynomials that are not $p$-type yet one of the iterates is $p$-type, i.e., eventually $p$-type polynomials. We fully characterize eventually $p$-type polynomials in Theorem \ref{thm: characterization_of_eventually_p-type_polynomials} and identify the least $p$-type iterate in Proposition \ref{prop: which iteration of an eventually p-type is p-type}. In \S \ref{sec3}, we discuss the properties of iterations of $p^r$-pure polynomials and discuss the conditions under which the composition of a $p^r$-pure polynomial and a $p$-type polynomial is $p^r$-pure. In \S \ref{sec4}, we use the results from \S \ref{sec3} together with a result from \cite{anuj} to conclude the eventual stability of $p^r$-pure polynomials. Moreover, we obtain some iterative techniques to produce irreducible polynomials from $p^r$-Dumas polynomials. Finally, in \S \ref{sec5}, we utilize the results in \S \ref{sec2} on eventually $p$-type polynomials to fully characterize a family of eventually 
$p^r$-pure polynomials.       

\subsection*{Acknowledgments} The authors would love to express their gratitude to Wade Hindes for reading an earlier draft of the manuscript and for several suggestions that helped the authors improve the manuscript. This work is supported by The Scientific and Technological Research Council of Turkey, T\"{U}B\.{I}TAK; research grant: ARDEB 1001/120F308. M. Sadek is supported by BAGEP Award of the Science Academy, Turkey.

\section{$p$-Type and Eventually $p$-Type Polynomials}
\label{sec2}
Throughout this article, we assume that $p$ is a rational prime. Moreover, all polynomials will be assumed to be in $\q[x]$ unless otherwise explicitly stated. 

In this section,
we introduce $p$-type and eventually $p$-type polynomials together with some of the properties of these polynomials.
For this purpose, we recall the definition of Gaussian valuations.
\begin{definition}\label{def: Gaussian Valuation}
Let $f(x)=a_dx^d+\ldots+a_0 \in \q[x]$ and $p$ be a prime. The Gaussian valuation of $f$ with respect to $p$ is defined by
\begin{equation*}
\displaystyle \nu_p(f):= \min_{0\leq i\leq d}\nu_p(a_i),
\end{equation*}
where $\nu_p(a_i)$ denotes the $p$-adic valuation of $a_i$.
\end{definition}
The abuse of notation may be justified by the fact that an element in $\q$ can be considered as a constant polynomial in $\q[x]$, hence the restriction of the Gaussian valuation with respect to $p$ over $\q$ is the $p$-adic valuation.

One sees easily that $\nu_p(f\cdot g)=\nu_p(f)+\nu_p(g)$ and $\nu_p(f+g)\ge \min(\nu_p(f),\nu_p(g))$ for $f,g\in\q[x]$.
\begin{definition}\label{def: p-type}
A polynomial $f(x)=a_dx^d+\ldots+a_0 \in \q[x]$ is said to be {\em $p$-type} if $\nu_p(a_d)=0$ and $f(x)\equiv a_dx^d \pmod p$. In other words, $\nu_p(a_0),\ldots, \nu_p(a_{d-1})\geq 1$.
\end{definition}
For example, a $p$-Eisenstein polynomial is $p$-type. 
\begin{definition}\label{def: eventually p-type}
Let $f\in \q[x]$ and $p$ be a prime. We say $f$ is {\em eventually $p$-type} if an iterate $f^n$ is $p$-type for some $n\geq 1$. 
\end{definition}
It is clear that a $p$-type polynomial is also eventually $p$-type. We are more interested in a polynomial which is not $p$-type but is eventually $p$-type. In other words, $f$ is not $p$-type but $f^n$ is $p$-type for some $n>1$. Consider the following example.
\begin{example}
The polynomial $f(x)=x^8+1$ is not $p$-type for any prime $p$, but
\begin{equation*}
f^2(x)=x^{64}+8 x^{56}+28 x^{48}+56 x^{40}+70 x^{32}+56 x^{24}+28 x^{16}+8 x^8+ 2  
\end{equation*}
is $2$-type. So, $f(x)$ is eventually $2$-type.
\end{example}
 In light of the previous example, it is valid to ask the following question.
 \begin{question}
  If a polynomial $f$ is not $p$-type yet it is eventually $p$-type, is there any restriction on the degree of $f$? Is there an exhaustive classification of such polynomials?
\end{question}
This question is answered in Theorem \ref{thm: characterization_of_eventually_p-type_polynomials}. We first need the following lemma.
\begin{lemma}\label{lemma: If f composed with g is p-type then f(x+g(0)) and g(x)-g(0) is p-type}
Suppose $f$ and $g$ are polynomials in $\q[x]$ such that $\nu_p(g)=\nu_p(f)=0$. If $f\circ g$ is $p$-type, 
then both $f\big(x+g(0)\big)$ and $g(x)-g(0)$ are $p$-type.  
\end{lemma}
\begin{proof}
We define the following polynomials
\begin{eqnarray*}
F(x)&:=& f\big(x+g(0)\big)=a_dx^d+ a_{d-1}x^{d-1}+\ldots+a_0, \\   
 G(x)&:=& g(x)-g(0)=b_ex^e+b_{e-1}x^{e-1}+\ldots+b_1x.   
\end{eqnarray*}
 It is clear that $\nu_p(a_d)=\nu_p(b_e)=0$, since otherwise $f\circ g$ would not be $p$-type. Let $i,j$, $0\le i\leq d$, $1\leq j\leq e$, be the least nonnegative integers such that $\nu_p(a_i)=\nu_p(b_j)=0$. If $i=d$, then $F$ is $p$-type, so we assume otherwise. 
 
 One has 
 \begin{equation*}
     F\left(G(x)\right)\equiv a_d{(b_ex^e+\ldots+b_jx^j)}^{d}+\ldots+a_i{(b_ex^e+\ldots+b_jx^j)}^{i}\pmod p.
 \end{equation*}
Note that in the expansion of $F(G(x))$ the monomial $a_ib_j^i{x^{ij}}$ is the monomial of the least degree whose coefficient has zero $p$-adic valuation. Since $F\circ G$ is $p$-type, it follows that there has to be another monomial in the expansion of $F(G(x))$ with coefficient of zero $p$-adic valuation and whose degree is still $ij$. However, any monomial in the expansion of $a_k{(b_ex^e+\ldots+b_jx^j)}^{k}$ where $k$ is such that $i< k\le d$ must be of degree at least $kj>ij$. Therefore, $i=d$ and $F$ is $p$-type. 

Based on the argument above, one has
\begin{equation*}
F\left(G(x)\right)\equiv a_d{(b_ex^e+\ldots+b_jx^j)}^{d}\pmod p.  
\end{equation*}
Since $F\circ G$ is $p$-type, this must yield that $j=e$, hence $G$ is $p$-type.   
 \end{proof}
 We are now in a place to prove the main result of this section.
\begin{theorem}\label{thm: characterization_of_eventually_p-type_polynomials}
 If $f\in \q[x]$ is not $p$-type but eventually $p$-type, then $f(x)= ax^{p^m}+ph(x)+b \in\q[x]$ for some $a,b\in\q$, $h\in\q[x]$ with $\nu_p(a)=\nu_p(b)=0$, $\deg(h)<p^m$ and $\nu_p(h)\ge 0$.
\end{theorem}
\begin{proof}
Assume $f(x)=a_dx^d+\ldots+a_0$ so that $f^n$ is $p$-type for some $n>1$. We will show that $\nu_p(f)=0$. It is obvious that $\nu_p(a_d)=0$, hence $\nu_p(f)\le 0$. We assume that $\nu_p(f)<0$. 

The polynomial $f^n-f^n(0)$ is clearly $p$-type. Given that $f$ divides $f^n-f^n(0)$ and $\nu_p\left(\frac{f^n-f^n(0)}{f}\right)\leq 0$, as the $p$-adic valuation of the leading coefficient of $f^n-f^n(0)$ is $0$, we obtain that 
$
  0= \nu_p\left(f^n-f^n(0)\right)=\nu_p(f)+ \nu_p\left(\frac{f^n-f^n(0)}{f}\right)<0
$, hence a contradiction. Thus, $\nu_p(f)=0$.

Since $f^n$ is $p$-type, we must have $\nu_p(f^{n-1})=0$. By Lemma~\ref{lemma: If f composed with g is p-type then f(x+g(0)) and g(x)-g(0) is p-type}, the polynomials $f\left(x+f^{n-1}(0)\right)$ and $f^{n-1}-f^{n-1}(0)$ are $p$-type, moreover $f^{n-1}\left(x+f(0)\right)$ and $f-f(0)$ are $p$-type. It follows that $\nu_p(a_0)=0$, else $f$ is $p$-type. Therefore, $f(x)\equiv a_dx^d+a_0 \pmod p$. In a similar fashion, $\nu_p(f^{n-1}(0))=0$. Now, knowing that $f\left(x+f^{n-1}(0)\right)\equiv a_d{\left(x+f^{n-1}(0)\right)}^d+a_0\equiv a_dx^d \pmod p$, we must have
\begin{eqnarray*}
 \nu_p\left(a_d (f^{n-1}(0))^d+a_0\right)\ge 1, \quad\textrm{ and }\quad
\nu_p \left(\binom{d}{k}\right) \ge 1 \textrm{ for all } 0<k<d.    
\end{eqnarray*}
In view of Kummer's Theorem, \cite[Definition 1.2]{Casacubertan}, the latter condition implies that $d=p^m$ for some $m\geq 1$. In conclusion, if $f$ is not $p$-type but eventually $p$-type, then $f(x)\equiv a_dx^{p^m}+a_0 \pmod p$ where $\nu_p(a_0)=\nu_p(a_d)=0$ as desired.   
\end{proof}
If a polynomial $f$ is not $p$-type but is eventually $p$-type, one is interested in the least integer $n>1$ such that $f^n$ is $p$-type. For $a\in\q$ with $\nu_p(a)=0$, we set $\ord_p(a)$ to be the multiplicative order of $a$ modulo $p$. The following proposition identifies such a minimal iterate.
\begin{proposition}\label{prop: which iteration of an eventually p-type is p-type}
Suppose $f(x)= ax^{p^m}+ph(x)+b \in\q[x]$ is such that $\nu_p(a)=\nu_p(b)=0$, $\deg(h)<p^m$ and $\nu_p(h)\ge 0$. 
Assume that $f(x)$ is an eventually $p$-type polynomial. Then the least integer $n>1$ such that $f^n$ is $p$-type is determined as follows:
\begin{itemize}
    \item[a)] $n=p$ if $a\equiv 1 \pmod p$  
    \item[b)] $n=\operatorname{ord}_p(a)$ otherwise.
\end{itemize}
\end{proposition}
We need the following lemma to prove Proposition \ref{prop: which iteration of an eventually p-type is p-type}. 

\begin{lemma}\label{lemma: iteration of a p-type mod p}
Suppose $f(x)= ax^{p^m}+ph(x)+b \in\q[x]$ is such that $\nu_p(a)=\nu_p(b)=0$, $\deg(h)<p^m$ and $\nu_p(h)\ge 0$. 
Then $f^n(x)\equiv a^nx^{p^{nm}}+b\sum_{i=0}^{n-1} a^i\pmod p$ for any $n\geq 1$. In particular, $f^n(0)\equiv b\sum_{i=0}^{n-1} a^i\pmod p$. 
\end{lemma}
\begin{proof}
A simple induction argument shows that $f^{n+1}(x)=f(f^n(x))\equiv a{(a^nx^{p^{nm}}+\sum_{i=0}^{n-1}a^ib)}^{p^m}+b\equiv a^{n+1}x^{p^{(n+1)m}}+b\sum_{i=0}^{n}a^i\pmod p$. 
\end{proof}

\begin{proof}[Proof of Proposition \ref{prop: which iteration of an eventually p-type is p-type}]
By Lemma~\ref{lemma: iteration of a p-type mod p}, if  $a\equiv 1\pmod p$, then $f^p(0)\equiv \sum_{i=0}^{p-1}b \equiv  pb\equiv 0\pmod p$; otherwise $f^k(0)\equiv b\sum_{i=0}^{k-1}a^i\equiv b(\frac{a^k-1}{a-1})\equiv 0\pmod p$. In either case, it is obvious that the specified $n$ is the smallest such integer.
\end{proof}
We end this section with an example.
\begin{example}
Consider 
\begin{equation*}
    f(x)=2x^5+\frac{5 x}{3}+7
\end{equation*}
Note that $\operatorname{ord}_5(2)=4$. One may see that
$f(x)\equiv 2x^5+2 \pmod 5$, $f^2(x)\equiv 4 x^{25}+1 \pmod 5$, $f^3(x)\equiv3 x^{125}+4 \pmod 5$
and $f^4(x) \equiv x^{625} \pmod 5$. Thus, although $f$ is not $5$-type, it is eventually $5$-type, and $f^4$ is $5$-type.
\end{example}

\section{$p^r$-Pure Polynomials}
\label{sec3}
In this section, we introduce $p^r$-pure polynomials, $r\ge1$. We show that the composition of two $p^r$-pure polynomials is also $p^r$-pure. Also, we prove that if $f$ is a $p^r$-pure polynomial and $g$ is $p$-type, then under certain conditions $f\circ g$ is $p^r$-pure. First, we define what a pure polynomial is.
\begin{definition}\label{def: pure polynomials}\cite{anuj}
A polynomial $f(x) = a_dx^d + a_{d-1}x^{d-1} + \ldots + a_0 \in \q[x]$ is said to be {\em $p^r$-pure} for some prime $p$ and some $r\geq 1$, if it satisfies all the following
\begin{itemize}
    \item[i)] $\nu_p(a_d)=0$,
    \item[ii)] $\nu_p(a_0)=r$,
    \item[iii)] $\frac{\nu_p(a_i)}{d-i}\geq \frac{r}{d}$ for all $1\leq i\leq d-1$.
\end{itemize}
\end{definition}
It is clear that a $p^r$-pure polynomial is $p$-type.

The previous definition can be interpreted using Newton polygons. We recall the definition of a Newton polygon. 
\begin{definition}\cite[Section 2.2.1]{prasolov_polynomials_2004}\label{def: newton polygon}
Let $f(x)=a_dx^d+\ldots+a_0 \in \q[x]$ with $a_da_0\neq 0$. For a prime $p$, suppose $\alpha_i=\nu_p(a_i)$. The Newton Polygon of $f$ with respect to $p$ is constructed as follows:
\begin{enumerate}
    \item[i)] Define  $S:=\{(0,\alpha_d),\ldots,(d-i,\alpha_i),\ldots, (d,\alpha_0)\}$. 
    \item[ii)] Consider the lower convex hull of $S$ to be $P_0=(0,\alpha_d),\ldots, P_r=(d,\alpha_0)$.
    \item[iii)] Construct a set of broken lines $P_0P_1,\ldots, P_{r-1}P_r$.
    \item[iv)] Mark the lattice points (points with integer coordinates) on the broken lines $P_0=Q_0,\ldots,P_r=Q_{r+s}$. They are called the \textit{vertices} of the Newton polygon. 
    \item[v)] The broken lines joining the vertices $Q_0Q_1,\ldots,Q_{r+s-1}Q_{r+s}$ are the \textit{sides} of the Newton polygon.
\end{enumerate} 
\end{definition}
Note that condition (iii) in Definition \ref{def: pure polynomials} is equivalent to saying that the Newton polygon of a $p^r$-pure polynomial consists of exactly one line segment joining $(0,0)$ and $(d,r)$. The family of $p^r$-Eisenstein polynomials provide an example of irreducible $p^r$-pure polynomials. However, a pure polynomial is not always irreducible; consider the following example.
\begin{example}\label{example: pure polynomials are not always irreducible}
The polynomial $f(x)=x^4+4=(x^2-2x+2)(x^2+2x+2)$ is $2^2$-pure but reducible over $\q$.
\begin{figure}[H]
    \centering
\begin{tikzpicture}
    \draw[step=1cm,gray,very thin] (0,0) grid (4,4);
	\draw (0,0) -- coordinate (x axis mid) (4,0);
    	\draw (0,0) -- coordinate (y axis mid) (0,4);
    	\foreach \x in {0,...,4}
     		\draw (\x,1pt) -- (\x,-1pt)
			node[anchor=north] {\x};
    	\foreach \y in {0,...,4}
     		\draw (1pt,\y) -- (-4pt,\y) 
     			node[anchor=east] {\y}; 
    \draw[thick,-] (0,0) -- (2,1);
    \draw[thick,-] (2,1) -- (4,2);
  \foreach \Point in {(0,0), (2,1), (4,2)}{
    \node at \Point {\textbullet};}
    \draw[-] (0,0) --  (2,1) ;
    \draw[-] (2,1) --  (4,2);
\end{tikzpicture}
\captionsetup{justification=centering}
\caption{The polynomial $f(x)$ is $2^2$-pure but reducible over $\q$.}
    \label{fig:pure reducible}
\end{figure}
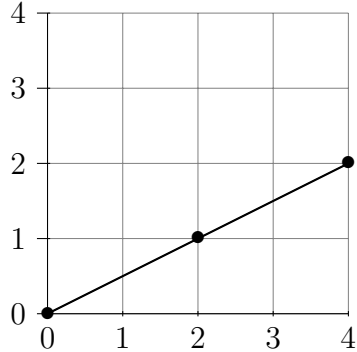
\end{example}
Nevertheless, if a $p^r$-pure polynomial satisfies Dumas criterion, then it must be irreducible over $\q$.
\begin{definition}\label{def: dumas criterion}\cite{Dumas}
A polynomial $f(x) = a_dx^d + a_{d-1}x^{d-1} + \ldots + a_0 \in \q[x]$ is called {\em $p^r$-Dumas} if there exists a prime $p$ and a positive integer $r$ such that
\begin{itemize}
    \item[i)] $\nu_p(a_d)=0$,
    \item[ii)] $\nu_p(a_0)=r$,
    \item[iii)] $\frac{\nu_p(a_i)}{d-i} \geq \frac{r}{d}$ for $1 \leq i \leq d-1$,
    \item[iv)] $\operatorname{gcd}(r,d)=1$.
\end{itemize}
 Moreover, a $p^r$-Dumas polynomial $f(x)$ is irreducible over $\q$.
 \end{definition}
 Observe that the $p$-Eisenstein criterion is a special case of the $p^r$-Dumas criterion with $r=1$. For $p$-Eisenstein polynomials, one can notice that all the iterates of a $p$-Eisenstein polynomial are $p$-Eisenstein. More generally, any $p$-Eisenstein polynomial enjoys the following property.
 \begin{proposition}\label{prop: If g is p-eisensten and f is p-type then fg is p-eisenstein}
Let $g,f \in \q[x]$. If $g$ is $p$-Eisenstein for some  prime $p$ and $f$ is $p$-type, then $g\circ f$ is $p$-Eisenstein. 
\end{proposition}
Next, we show that the property in Proposition \ref{prop: If g is p-eisensten and f is p-type then fg is p-eisenstein} holds for certain families of $p^r$-pure polynomials. Nevertheless, some preparation is needed.
\begin{lemma}\label{lem: f(c)_where_f_is_pure}
Suppose $f$ is a $p^r$-pure polynomial of degree $d$ and let $c\in \q$ be such that $\nu_p(c)>\frac{r}{d}$. Then $\nu_p(f(c))=r$.
\end{lemma}
\begin{proof}
We set $s:=\nu_p(c)$. Suppose $f(x)=a_dx^d+\ldots+a_0$ is a $p^r$-pure polynomial. Assume that $\nu_p(c)=s>\frac{r}{d}$. We see that $\nu_p(a_dc^d)=ds>d(\frac{r}{d})=r$. Also, $\nu_p(a_ic^i)=\nu_p(a_i)+is$ where $0<i<d$. Since $f$ is $p^r$-pure, one has $\nu_p(a_i)+is>(d-i)\frac{r}{d}+i\frac{r}{d}=r$. Given that $\nu_p(a_0)=r$ and that $f(c)=a_d{c}^d+\ldots+a_i{c}^i+\ldots+a_0$, it follows that $\nu_p(f(c))=r$. 
\end{proof}
Next, we extend Proposition~\ref{prop: If g is p-eisensten and f is p-type then fg is p-eisenstein} to $p^r$-pure polynomials in the following theorem. 
\begin{theorem}\label{thm: Composition_of_pure_and_p-type}
Let $f\in\q[x]$ be a $p^r$-pure polynomial of degree $d>1$. Suppose that  $g(x)=bx^e+p^sh(x)\in\q[x]$, $e\ge1$, is such that $\nu_p(b)=0$, $\deg(h)\le e-1$, $\nu_p(h)\ge 0$ and $s>\frac{r}{d}$.
Then $f\circ g$ is $p^r$-pure.
\end{theorem} 
\begin{proof}
Let $f(x)=a_dx^d+\ldots+a_ix^i+\ldots+a_0$ be $p^r$-pure. To show that $f\circ g$ is $p^r$-pure, we need to prove that $f\circ g$ satisfies the conditions of Definition~\ref{def: pure polynomials}. Since the $p$-adic valuation of the leading coefficients of both $f$ and $g$ is zero, the leading coefficient of $f\circ g$ has zero $p$-adic valuation, hence i) of Definition~\ref{def: pure polynomials} is satisfied. In view of Lemma~\ref{lem: f(c)_where_f_is_pure}, as $\nu_p(g(0))\geq s>\frac{r}{d}$, one sees that  $\nu_p(f(g(0))=r$, hence ii) is satisfied. 

We are now left with showing that $f\circ g$ satisfies iii) of Definition~\ref{def: pure polynomials}.  We will show that every monomial in the expansion of $f(g(x))$ satisfies iii). In fact, a monomial in the expansion of $a_d{\left(bx^e+p^s(h(x)\right)}^d=a_d\sum_{k=0}^d\binom{d}{k}{(bx^e)}^{d-k}{\left(p^sh(x)\right)}^{k}$ is of the form $c_{de-ek+k\alpha}x^{de-ek+k\alpha}$, where
\begin{equation*}
c_{de-ek+k\alpha}=\binom{d}{k}a_d{b}^{d-k}p^{sk}h_{\alpha},\qquad \textrm{for some } h_{\alpha}\in \q\textrm{ with }\nu_p(h_{\alpha})\ge 0,
\end{equation*}
and $0\leq \alpha\leq e-1$.
One may see that $\nu_p(c_{de-ek+k\alpha})\geq ks>k\frac{r}{d}$. We claim that $\nu_p(c_{de-ek+k\alpha})\ge [de-(de-ek+k\alpha)]r/(de)$, since otherwise
$ [de-(de-ek+k\alpha)]r/(de)>kr/d$, i.e., $e-\alpha>e$ which contradicts the fact that $\alpha$ is a nonnegative integer.

In a similar fashion, a monomial in the expansion of $a_i{(bx^e+p^sh(x))}^i=a_i\sum_{k=0}^i \binom{i}{k}(bx^e)^{i-k}{\left(p^sh(x)\right)}^k$, $0<i<d$, is of the form $t_{ie-ek+k\alpha}x^{ie-ek+k\alpha}$, where
\begin{equation*}
t_{ie-ek+k\alpha}=\binom{i}{k}a_i{b}^{i-k}p^{sk}h_{\alpha},\qquad \textrm{for some } h_{\alpha}\in \q\textrm{ with }\nu_p(h_{\alpha})\ge 0,
\end{equation*}
and $0\leq \alpha\leq e-1$. One sees that $\nu_p(t_{ie-ek+k\alpha})\geq \nu_p(a_i)+ks>\nu_p(a_i)+k\frac{r}{d}$.
Again, we claim that $\nu_p(t_{ie-ek+k\alpha})\ge [de-(ie-ek+k\alpha)]r/(de)$, since otherwise we use the fact that $f$ is $p^r$-pure, in particular $\frac{\nu_p(a_i)}{d-i}\geq\frac{r}{d}$, to obtain that $[de-(ie-ek+k\alpha)]r/(de)>\nu_p(a_i)+kr/d>(d-i+k)r/d$. The latter leads to the contradiction $k\alpha<0$.

This concludes the proof as $f\circ g$ satisfies iii) of Definition~\ref{def: pure polynomials}
\end{proof}

In particular, if $d>r$ in the previous theorem, we get the following interesting corollary.
\begin{corollary}\label{cor: Composition of pure with p-type when when r/d<1}
Let $f$ be a $p^r$-pure polynomial of  degree $d>r$. If $g$ is a $p$-type polynomial, then $f\circ g$ is $p^r$-pure.
\end{corollary}
\begin{proof}
This is a special case of Theorem \ref{thm: Composition_of_pure_and_p-type} with $s\geq 1>\frac{r}{d}$. 
\end{proof}

We now prove the following lemma. \begin{lemma}[Purity Lemma] \label{lem: purity_lemma}
Let $f$ be a $p^r$-pure polynomial of degree $d>1$ and $g(x)=bx^t$, $t\ge 1$, be a monomial in $\q[x]$. Then the following statements hold true.
\begin{itemize}
    \item[i)] If $\nu_p(b)=0$, then $f\circ g$ is $p^r$-pure and there exists some $c\in \q$ such that $g\circ f+c$ is  $p^r$-pure. 
    \item[ii)] If $\frac{\nu_p(b)}{e-t}\geq \frac{r}{e}$ for some $e>t$, then there exists some $c\in \q$ such that $x^{de}+g\circ f+c$ is $p^r$-pure.
\end{itemize}
\end{lemma}
\begin{proof}
Let $f(x)=a_dx^d+\ldots+a_0$ and $\nu_p(b)=0$. Since
\begin{equation*}
    f(bx^t)=\sum_{i=0}^da_ib^ix^{it},
\end{equation*}
it is easily seen that the $p$-adic valuation of the leading coefficient is zero and the $p$-adic valuation of the constant coefficient is $\nu_p(a_0)=r$. Moreover, $f(bx^t)$ is a polynomial of degree $dt$. Indeed, for any $0<i< d$, one observes that $\nu_p(a_ib^i)=\nu_p(a_i)$. Since
$\frac{\nu_p(a_i)}{d-i}\geq \frac{r}{d}$, it follows that
 $\frac{\nu_p(a_i)}{dt-it}\geq \frac{r}{dt}$, hence the first part of i) is proved.

 We now write $g(f(x))= b{\left(a_dx^d+\ldots+a_0\right)}^t$. A monomial in the latter expansion is of the form  
\[ bh\prod_{i=0}^{d}{a_i}^{q_i}x^{iq_i},\quad\textrm{for some }h\in\q\textrm{ with }\nu_p(h)\ge0,\] 
where $\sum_i q_i=t$.
If $1\leq \sum_i iq_i< dt$, one observes that
\begin{equation*}
    \nu_p\left(\prod_{i=0}^{d}a_i^{q_i}\right)=\sum_{i=0}^{d}q_i\nu_p(a_i)\geq \sum_{i=0}^dq_i\frac{r}{d}(d-i).
\end{equation*}
In order to show that such a monomial satisfies condition iii) in Definition \ref{def: pure polynomials}, one must have
\begin{equation*}
  \frac{\nu_p\left(\prod_{i=0}^{d}a_i^{q_i}\right)}{dt-\sum_iiq_i}\geq \frac{r}{dt}.
\end{equation*}
We assume on the contrary that the latter inequality does not hold. In particular, one has
\begin{equation*}
\frac{\sum_{i}q_i\frac{r}{d}(d-i)}{dt-\sum_iiq_i}\leq \frac{\nu_p\left(\prod_{i}a_i^{q_i}\right)}{dt-\sum_iiq_i}<\frac{r}{dt}.
\end{equation*}
Thus, one obtains
\begin{equation*}
\frac{\sum_{i}q_i(d-i)}{dt-\sum_iiq_i}<\frac{1}{t}.
\end{equation*}
Given that $\sum_iq_i=t$, we get the following contradiction
\begin{equation*}
    1=\frac{dt-\sum_iiq_i}{dt-\sum_iiq_i}<\frac{1}{t}
    \end{equation*}
It is easy to see that the $p$-adic valuation of the constant coefficient of $b{(f(x))}^t-b{f(0)}^t+p^r$ is exactly $r$. Thus, $g(f(x))-b{(f(0))}^t+p^r$ is $p^r$-pure.

For ii), given that $\frac{\nu_p(b)}{e-t}\geq \frac{r}{e}$ for some $e>t$, if $1\leq \sum_i iq_i< dt$, then one gets
\begin{equation*}
 \nu_p\left(b\prod_{i=0}^{d}a_i^{q_i}\right)=\nu_p(b)+\sum_{i=0}^dq_i\nu_p(a_i) \geq \frac{r}{e}(e-t)+\sum_{i=0}^dq_i\frac{r}{d}(d-i)=\frac{r}{e}(e-t)+rt-\frac{r}{d}\sum_{i=0}^diq_i.    
\end{equation*}
Again one claims that $x^{de}+g(f(x))$ satisfies condition iii) in Definition \ref{def: pure polynomials}. More precisely,  we will show that $\nu_p\left(b\prod_{i=1}^{d}a_i^{q_i}\right)\ge (de-\sum iq_i)r/(de)$, since otherwise 
\begin{equation*}
    \frac{r}{e}(e-t)+rt-\frac{r}{d}\sum_{i=0}^d iq_i<\frac{(de-\sum iq_i)r}{de}.
\end{equation*}
Simplification yields the following contradiction 
$\sum_i iq_i>dt.$

We note that for the monomial $ba_d^tx^{dt}$, one has 
\begin{equation*}
    \frac{\nu_p(ba^t)}{e-t}= \frac{\nu_p(b)}{e-t}\geq \frac{r}{e}, \text{hence } \frac{\nu_p(ba^t)}{de-dt}\ge\frac{r}{de}.
\end{equation*}
It follows that the polynomial $x^{de}+g(f(x))$ satisfies i) and iii) of Definition~\ref{def: pure polynomials}. Therefore, $x^{de}+g(f(x))-g(f(0))+p^r$ is a $p^r$-pure polynomial.
\end{proof}
We are now ready to prove the following result. 
If the polynomial $g\in\q[x]$ in Theorem \ref{thm: Composition_of_pure_and_p-type} is $p^r$-pure, we conclude that $p^r$-pure polynomials are closed under composition.  
\begin{theorem}\label{thm: the_composition_of _two_pure_polynomials_is_pure}
If $f, g \in \q[x]$ are $p^r$-pure with $\operatorname{deg}(f)>1$, then $f\circ g$ is $p^r$-pure.
\end{theorem}
\begin{proof}
 If $f(x)=a_dx^d+\ldots+a_0$ and $g(x)=b_ex^e+\ldots+b_0$ are $p^r$-pure polynomials with $d>1$ and $e\geq1$, then each monomial in the expansion $a_i{(b_ex^e+\ldots+b_0})^i$, $0<i\leq d$, satisfies iii) in Definition~\ref{def: pure polynomials}, see Lemma~\ref{lem: purity_lemma}. It is obvious that the $p$-adic valuation of the leading coefficient is zero. Finally, by Lemma \ref{lem: f(c)_where_f_is_pure}, since $d>1$ and $\nu_p(g(0))=r>\frac{r}{d}$, then, $\nu_p(f(g(0)))=r$ and $f\circ g$ is a $p^r$-pure polynomial. 
\end{proof}

The following corollary follows directly from Theorem \ref{thm: the_composition_of _two_pure_polynomials_is_pure}.
\begin{corollary}\label{cor: iterates of a pure polynomial are pure}
If $f$ is a $p^r$-pure polynomial with $\deg f>1$, then  $f^n$ is $p^r$-pure for all $n\geq 1$.
\end{corollary}
 Theorem~\ref{thm: Composition_of_pure_and_p-type} and Theorem~\ref{thm: the_composition_of _two_pure_polynomials_is_pure} can be used to prove $p^r$-purity of $f\circ g$ for different classes of pairs of polynomials $f$, $g$. In other words, for given polynomials $f$, $g$, Theorem~\ref{thm: Composition_of_pure_and_p-type} can be successfully used to show that $f\circ g$ is $p^r$-pure whereas either $f$ or $g$ fails to satisfy the hypothesis of Theorem~\ref{thm: the_composition_of _two_pure_polynomials_is_pure}, and vice versa. This can be illustrated by the following example.
\begin{example}
The following polynomials
\begin{align*}
    f(x)&=x^2+32,\\
    g(x)&= x^4+4x^3+32
\end{align*}
are both $2^5$-Dumas. The polynomial
\begin{align*}
    f\left(g(x)\right)=x^8+8x^7+16x^6+64 x^4+256 x^3+1056
    \equiv x^8+8 x^7+16 x^6 \pmod{2^5}
\end{align*}
is $2^5$-Dumas too as $\frac{\nu_2(8)}{8-7}, \frac{\nu_2(16)}{8-6}>\frac{5}{8}$. This is a direct application of Theorem~\ref{thm: the_composition_of _two_pure_polynomials_is_pure}. However, if we try to apply Theorem \ref{thm: Composition_of_pure_and_p-type}, then $d=\operatorname{deg}(f)=2$, $r=5$, so $s>\frac{5}{2}$. However, $g(x)\equiv x^4+4x^3\not \equiv x^4 \pmod{2^3}$. Thus, Theorem~\ref{thm: Composition_of_pure_and_p-type} fails to prove $f\circ g$ is $2^5$-pure in this case. 

Now, if we introduce the polynomial
\begin{equation*}
    h(x)=x^4+8,
\end{equation*}
we get $h(x)\equiv x^4 \pmod{2^3}$. According to Theorem~\ref{thm: Composition_of_pure_and_p-type}, $f\circ h$ is $2^5$-Dumas. In fact,
\begin{equation*}
   f\left(h(x)\right)=x^8+16 x^4+96 
\end{equation*}
and $\nu_2\left(f(h(0))\right)=\nu_2(96)=5$ and $\frac{\nu_2(16)}{8-2}=\frac{2}{3}>\frac{5}{8}$. Nevertheless, $h$ is not $2^5$-pure and Theorem~\ref{thm: the_composition_of _two_pure_polynomials_is_pure} can not be applied.
\end{example}

\section{Dynamical irreducibility and eventual stability of families of polynomials}
\label{sec4}
In this section, we will discuss several applications of Theorems ~\ref{thm: Composition_of_pure_and_p-type} and \ref{thm: the_composition_of _two_pure_polynomials_is_pure} to arithmetic dynamics. In the previous section, we introduced Dumas polynomials as a class of irreducible pure polynomials. The following corollary follows directly from Theorem~\ref{thm: the_composition_of _two_pure_polynomials_is_pure}.
\begin{corollary}\label{cor: dumas polynomials are dynamically irreducible}
Let $f$ and $g$ be $p^r$-Dumas polynomials in $\q[x]$, then $f\circ g$ is $p^r$-Dumas. In particular, a $p^r$-Dumas polynomial is dynamically irreducible over $\q$. 
\end{corollary}
\begin{example}
Consider the following trinomial in $\q[x]$
\begin{equation*}
f(x)= x^{d}+ax^{d-1}+p^{2^k};  \text{ $d$ is odd, } k\geq 0, \text{ and } \nu_p(a)> \frac{2^k}{d}.    
\end{equation*}
Note that $\frac{\nu_p(a)}{d-(d-1)}=\nu_p(a)>\frac{2^k}{d}$ and $\operatorname{gcd}(2^k,d)=1$. In this case, $f$ is $p^{2^k}$-Dumas and thus dynamically irreducible over $\q$.
\end{example}
\begin{corollary}\label{cor: composition of dumas and p-type}
Let $g$ be a $p^r$-Dumas polynomial of degree $d$ and $f(x)= ax^e+ p^s h(x)\in\q[x]$ be such that $h(x)\in\q[x]$ with $\nu_p(a)=0$, $\deg(h)<e$, $\nu_p(h)\ge 0$, and $s>\frac{r}{d}$. If $\gcd(r,e)=1$, then $g^n\circ f^m$ is irreducible for all $n,m\geq 1$. In particular, $g^n$ is $f$-stable for any $n\geq 1$. 
\end{corollary}
\begin{proof}
By Corollary~\ref{cor: dumas polynomials are dynamically irreducible}, $g^n$ is $p^r$-Dumas for any $n\geq 1$. Also, $f^m(x)\equiv bx^{e^m} \pmod{p^s}$ where $b\in\Q$ is such that $\nu_p(b)=0$. Thus, $f^m$ is $p$-type. By virtue of Theorem~\ref{thm: Composition_of_pure_and_p-type}, since $\operatorname{gcd}(r,e)=1$, one has $g^n\circ f^m$ is $p^r$-Dumas.
\end{proof}
\begin{example}
Set 
\begin{align*}
    f(x)&=x^{17}+27 x^{12}+27 x^{10}+162 x^7+729 x^5+4374\\
    &=(x^7+27)(x^5+9)(x^5+18).
\end{align*}
Note that $g_1(x)=x^7+27$ is $3^3$-Dumas with $\frac{r_1}{d_1}=\frac{3}{7}<1$, $g_2(x)=x^5+9$ is $3^2$-Dumas with $\frac{r_2}{d_2}=\frac{2}{5}<1$ and $g_3(x)=x^5+18$ is $3^2$-Dumas with  $\frac{r_3}{d_3}=\frac{2}{5}<1$. This implies that $s_1,s_2,s_3\geq 1$. Since $f(x)\equiv x^{17}\pmod{3}$ and $\gcd(17,7)=\gcd(17,5)=1$, by Corollary \ref{cor: composition of dumas and p-type}, the polynomials $g_1$, $g_2$ and $g_3$ are $f$-stable. In other words, for any $n\geq 1$, the number of irreducible factors of $f^n$ is exactly $3$.
\end{example}
The previous example motivates the following corollary.
\begin{corollary}\label{cor: if f^n is the product of dumas polynomials with some conditions then all the factors are f-stable}
Let $f(x)= ax^e+ p^s h(x)\in\q[x]$ be such that $h(x)\in\q[x]$ with $\nu_p(a)=0$, $\deg(h)<e$, $\nu_p(h)\ge 0$ and $s\ge1$.
 For $n\geq 1$, assume that $f^n(x)=g_1(x)g_2(x)\cdots g_t(x)$ where $g_i$ is irreducible of degree $d_i\ge 1$, $1\le i\le t$. 
 
 If for all $1\leq i\leq t$, the following conditions hold
\begin{itemize}
    \item[i)] $g_i$ is $p^{r_i}$-Dumas for some $r_i\geq 1$,
    \item[ii)] $\gcd(r_i,e)=1$,
    \item[iii)] $s>\frac{r_i}{d_i}$,
\end{itemize}
then $g_1,g_2,\ldots, g_t$ are all $f$-stable. In fact, for any $N\geq n$, the number of irreducible factors of $f^N$ is exactly $t$. Moreover, the irreducible factors $G_i$ of $f^N$ are $p^{r_i}$-Dumas, $1\le i\le t$. 
\end{corollary}
\begin{proof}
 In view of Corollary \ref{cor: composition of dumas and p-type}, one sees that $g_i^n$, $1\le i\le t$, are $f$-stable for any $n\ge 1$. In particular, for any $N\ge n$, one obtains $f^N=f^n(f^{N-n})=g_1\left(f^{N-n}\right) g_2\left(f^{N-n}\right)\cdots g_t\left(f^{N-n}\right)$ where each factor $g_i\left(f^{N-n}\right)$ is irreducible. Moreover, given that $g_i$ is $p^{r_i}$-Dumas and $\gcd(r_i,d_i)=1$, Theorem~\ref{thm: Composition_of_pure_and_p-type} asserts that $g_i\left(f^{N-n}\right)$ is $p^{r_i}$-Dumas, $1\le i \le t$.    
\end{proof}
 One may drop the condition “$s>\frac{r}{d}$” in Corollary~\ref{cor: composition of dumas and p-type} to obtain the following result. 
\begin{corollary}\label{cor: the iterate of dumas when composed with p-type is dumas}
Let $g$ be a $p^r$-Dumas polynomial of degree $d$ and $f(x)= ax^e+ p^s h(x)\in\q[x]$ be such that $h(x)\in\q[x]$ with $\nu_p(a)=0$, $\deg(h)<e$, $\nu_p(h)\ge 0$, and $\gcd(r,e)=1$. There exists an integer $N\geq 1$ such that for all $n\geq N$, $g^n\circ f^k$ is irreducible for all $k\geq 1$. In particular, $g^n$ is $f$-stable for all $n\geq N$.   
\end{corollary}
\begin{proof}
Let $N=\min\{n: s>\frac{r}{d^n}\}$. For any $n\geq N$, $g^n$ is $p^r$-Dumas of degree $d^n$ by Corollary~\ref{cor: dumas polynomials are dynamically irreducible}. Since $s>\frac{r}{d^N}\geq \frac{r}{d^n}$, Corollary~\ref{cor: composition of dumas and p-type} implies that $g^n\circ f^k$ is $p^r$-Dumas for all $n\geq N$ and $k\geq 1$.
\end{proof}
In Corollary \ref{cor: the iterate of dumas when composed with p-type is dumas}, when $n<N$, the irreducibility of $g^n\circ f$  is not guaranteed. We consider the following example.
\begin{example}
The following polynomials are $3^3$-Dumas and $3$-type, respectively
 \begin{align*}
   g(x)&= x^2+27,\\
   f(x)&=x^2+3x+3.   
 \end{align*}
  We notice that the polynomial $f$ does not satisfy the hypothesis of Corollary \ref{cor: composition of dumas and p-type}. Based on Corollary~\ref{cor: the iterate of dumas when composed with p-type is dumas},  $N=\min\{n: 1>\frac{3}{2^n}\}=2$. Therefore, $g^n\circ f$ is $3^3$-Dumas for all $n\geq 2$. Indeed, 
  \begin{equation*}
     (g^2\circ f)(x)= x^8+12 x^7+66 x^6+54 x^5+27 x^4+27 x^2+27 
  \end{equation*}
 is $3^3$-Dumas and by Corollary~\ref{cor: dumas polynomials are dynamically irreducible}, $g^n\circ f$ is $3^3$-Dumas. Yet,
 \begin{equation*}
 g\circ f=\left(x^2+3\right) \left(x^2+6 x+12\right) 
  \end{equation*}  
  is reducible.
\end{example}
We have discussed the applications of Theorems~\ref{thm: Composition_of_pure_and_p-type} and \ref{thm: the_composition_of _two_pure_polynomials_is_pure} in constructing $p^r$-Dumas polynomials. We recall that $p^r$-pure polynomials are not always irreducible. This motivates questioning the existence of an upper bound on the number of irreducible factors of $p^r$-pure polynomials.  
\begin{proposition}\cite[Theorem 1.2]{anuj}\label{anuj}
Let $f$ be a $p^r$-pure polynomial of degree $d$ in $\q[x]$. Then $f$ has at most $\gcd(d,r)$ irreducible factors over $\q$ and each irreducible factor has degree at least $\frac{d}{\gcd(d,r)}$.
\end{proposition}
Dynamically, we can conclude the following result regarding the upper bound on the number of irreducible factors of an iteration of a $p^r$-pure polynomial. 
\begin{theorem}\label{pure polynomials are eventually stable}
Suppose that $f\in\q[x]$ is a $p^r$-pure polynomial of degree $d$. Then for any $n\geq 1$, the iterate $f^n$ has at most $\gcd(d^n,r)$ irreducible factors over $\q$ and each irreducible factor has degree at least $\frac{d^n}{\gcd(d^n,r)}$. Moreover, $f$ is eventually stable over $\q$. 
\end{theorem}
\begin{proof}
 By Corollary \ref{cor: iterates of a pure polynomial are pure}, the iterate $f^n$ is $p^r$-pure and by Proposition \ref{anuj} it has at most $\gcd(d^n,r)$ irreducible factors over $\q$ and each irreducible factor has degree at least $\frac{d^n}{\gcd(d^n,r)}$. Moreover, let $c_n=\gcd(d^n,r)$ and define $k_n$ to be the number of irreducible factors of $f^n$. Observe that the set $\{c_1,\ldots,c_n,\ldots\}$ is finite as $c_n\le r$. Therefore, there must exist an $N\geq 1$ such that for all $n\geq N$, one has $\gcd(d^n,r)=c_n=c_N=\gcd(d^N,r)$. In particular, one obtains that $k_n\leq c_n\leq c_N$ for all $n\geq 1$. It follows that the number of irreducible factors $k_n$ of $f^n$ is at most $c_N$ for all $n\ge 1$, hence $f $ is eventually stable.  
\end{proof}
Observe that Corollary 6 in \cite{hamblen} (except the case when $c=1$) follows as a corollary of Theorem~\ref{pure polynomials are eventually stable} and Corollary~\ref{cor: dumas polynomials are dynamically irreducible}  
\begin{corollary}\label{cor: x^d+c}
Let $f(x)=x^d+c \in \q[x]$. Then $f$ is eventually stable whenever $c\neq 0$ is not the reciprocal of an integer.
\end{corollary}
\begin{proof}
let $c=\frac{a}{b}$ such that $a\neq \pm 1$, $b\neq 0$ and $\operatorname{gcd}(a,b)=1$. There exists a prime $p$ such that $\nu_p(c)>0$. By Theorem \ref{pure polynomials are eventually stable}, $f$ is $p^{\nu_p(c)}$-pure, hence eventually stable.  
\end{proof}
\begin{Remark}
In Corollary \ref{cor: x^d+c}, any iterate $f^n$ has at most $\max\{\gcd(\nu_p(c),d^m):m\geq 1\}$ irreducible factors. In particular, if $\gcd(\nu_p(c),d)=1$, then $f$ is dynamically irreducible by Corollary~\ref{cor: dumas polynomials are dynamically irreducible}.
\end{Remark}
The following definition was introduced in \cite{HeathBrown2019} for polynomials defined over a finite field. 
\begin{definition}
Let $I$ be a set of polynomials in $\q[x]$ with positive degrees. We say $I$ is a {\em dynamically irreducible set} in $\q[x]$ if any polynomial formed by composition of polynomials in $I$ is irreducible over $\q$. 
\end{definition}
One notices that our work up to this point has focused on dynamically irreducible sets of the form $I=\{f\}$ where $f\in\q[x]$. The set of  $p$-Eisenstein polynomials for a particular prime $p$ of degree at least $2$ is another example of a dynamically irreducible set. In light of our results, we display the following example. 
\begin{example}
Let $p$ and $q$ be rational primes. Define 
\begin{equation*}
  E(p):= \{f\in \q[x]: f \text{ is } p\text{-Eisenstein, } p \nmid \operatorname{deg}(f) \text{ and } \operatorname{deg}(f)>1\}.  
\end{equation*}
 The set $E(p)$ is  dynamically irreducible over $\q$. Also, define
\begin{equation*}
    D(p,q):=\{f\in \q[x]: f \text{ is } p^{q^k}\text{-Dumas with } \deg(f)>q^k \textrm{ for some } k\geq 1\}.
\end{equation*}
  In fact, $D(p,q)$ is a dynamically irreducible set because if $f$ and $g$ are $p^{q^k}$-Dumas and $p^{q^m}$-Dumas respectively, Corollary \ref{cor: Composition of pure with p-type when when r/d<1} ensures that the composition $f\circ g$ (respectively, $g\circ f$) is $p^{q^k}$-Dumas (respectively, $ p^{q^m}$-Dumas). 
  
  Moreover, the set $E(p)\cup D(p,q)$ is also dynamically irreducible because if $f\in D(p,q)$ (respectively, $ f\in E(p)$) and $g \in E(p)$  (respectively, $g\in D(p,q)$), then $f\circ g \in D(p,q)$  (respectively, $f\circ g\in E(p)$), see Corollary \ref{cor: Composition of pure with p-type when when r/d<1} (respectively, Proposition \ref{prop: If g is p-eisensten and f is p-type then fg is p-eisenstein}). 
\end{example}

We can extend this definition further for eventually stable polynomials.
\begin{definition}\label{def: eventually_stable_set}
 We say $S$ is an {\em eventually stable set} $^1$\footnote{\textsuperscript{1}Thanks to Wade Hindes for suggesting this definition.} in $\q[x]$ if there exists $c\ge 1$ such that the number of irreducible factors of any polynomial formed by composition of polynomials in $S$ is at most $c$.
\end{definition}
Based on our results, an example of an eventually stable set is the following.
\begin{example}
Let $p$ be a prime and $R$ be a finite set of positive integers. Define 
\begin{equation*}
    S(p,R):=\{f\in \q[x]: f \text{ is }p^{r}\text{-pure for some } r\in R \text{ such that } \deg(f)>r\}.  
\end{equation*}
If $f$ is $p^{r_1}$-pure and $g$ is $p^{r_2}$-pure such that $\deg(f)>r_1$ and $\deg(g)>r_2$, we know from Corollary~\ref{cor: Composition of pure with p-type when when r/d<1} that the composition $f\circ g$ (respectively, $g\circ f$) is $p^{r_1}$-pure (respectively, $p^{r_2}$-pure). Also, Theorem~\ref{thm: the_composition_of _two_pure_polynomials_is_pure} ensures that the iterates of $f$ (respectively, $g$) are $p^{r_1}$-pure (respectively, $p^{r_2}$-pure). It follows that the number of irreducible factors of any arbitrary composition is at most $\max(R)$, see Proposition \ref{anuj}.
\end{example}
A consequence of Definition~\ref{def: eventually_stable_set} and Corollary~\ref{cor: the iterate of dumas when composed with p-type is dumas} is the following.
\begin{corollary}\label{cor: finite_set_pure_polnomial_induces_eventually_stable_set}
For a fixed prime $p$, let $F$ be any set of $p^{r_f}$-pure polynomials $f$ such that the set $\{r_f:f\in F\}$ is finite. Then there exists an $N\geq 1$ such that the set $F^N:=\{f^N:f\in F\}$ is an eventually stable set. In particular, If all polynomials $f\in F$ are $p^{r_f}$-Dumas polynomials such that $\gcd\left(\deg(f),r_g\right)=1$ for any $f,g\in F$, then $F^N$ is a dynamically irreducible set.  
\end{corollary}
\begin{proof}
Define $s_f:=\nu_p(f(x)-ax^{d_f})$ where $a$ is the leading coefficient of $f\in F$ whose degree is $d_f$.
Set ~$r=\max\{r_f:f\in F\}$, $s=\min\{s_f: f\in F\}$ and $d=\min\{d_f: f\in F\}$. Let $N \ge 1$ be the least integer such that $s>\frac{r}{d^N}$. Note that for any $f\in F$, $s_f\geq s>\frac{r}{d^N}\geq \frac{r_f}{d_f^N}$. Indeed, the set $F^N=\{f^N:f\in F\}$ is an eventually stable set and the number of irreducible factors of any arbitrary composition is at most $r$. Suppose $f^N\circ G$ is an arbitrary composition of polynomials in $F^N$ where the degree of $G$ is $D$. We know that $f^N$ is $p^{r_f}$-pure by Corollary \ref{cor: iterates of a pure polynomial are pure}. Assume $\alpha=\nu_p(G(x))-bx^D$ such that $b$ is the leading coefficient of $G$. We have $\alpha\geq s>\frac{r}{d^N}\geq\frac{r_f}{d_f^N}$. Using Corollary \ref{thm: Composition_of_pure_and_p-type}, $f^N\circ G$ is $p^{r_f}$-pure with at most $\gcd(r_f,d_f^ND)\leq r_f\leq r$ irreducible factors and hence $F^N$ is an eventually stable set. If any $f\in F$ is $p^{r_f}$-Dumas and $\gcd(d_f,r_g)=1$ for any $g\in F$, then it is easy to see that the degree of $f^N\circ G$ is relatively prime to $r_g$ for any $f,g\in F$. It follows that $f^N\circ G$ is $p^{r_f}$-Dumas. Thus, $F^N$ is a dynamically irreducible set. 
\end{proof}
\section{Eventually $p^r$-Pure Polynomials}
\label{sec5}
In this section, we discuss polynomials that are not $p^r$-pure but one of the iterates is $p^r$-pure. Consider the following example. 
\begin{example}
The polynomial
\begin{equation*}
    f(x)=-x^3-\frac{39 x^2}{7}-\frac{72 x}{7}-\frac{31}{35}
\end{equation*}
is not $p$-type for any prime $p$. Yet,
\begin{align*}
    f^2(x)&=x^9+\frac{54 x^8}{7}+\frac{1287 x^7}{49}+\frac{56607 x^6}{1715}-\frac{53919 x^5}{1715}-\frac{36864 x^4}{245}\\&-\frac{696429 x^3}{8575}+\frac{1465479 x^2}{8575}+\frac{356184 x}{1715}-\frac{1090557}{6125}
\end{align*}
is $3^3$-pure.
\end{example}
The previous example motivates the following definition.
\begin{definition}\label{def: eventually_pure}
Let $p$ be a prime and $r$ be a positive integer. A polynomial $f\in \q[x]$ of degree $d$ is said to be {\em eventually $p^r$-pure} if $f^n$ is $p^r$-pure for some $n\ge 1$. Similarly, $f$ is {\em eventually $p^r$-Dumas} if $f^n$ is $p^r$-Dumas for some $n\ge 1$.
\end{definition}
The following corollary follows directly from Theorem \ref{pure polynomials are eventually stable} and \cite[Lemma 1.2]{Odoni}. 
\begin{corollary}
\label{cor1}
An eventually $p^r$-pure polynomial is eventually stable. In particular, an eventually $p^r$-Dumas polynomial is dynamically irreducible.  
\end{corollary}
Our aim is to provide a complete characterization of eventually $p^r$-pure polynomials of degree $d>r$. First, we introduce the following proposition.
\begin{proposition}\label{fg is pure then f(x+c)is pure}
Let $r$ be a positive integer and $p$ be a prime. Suppose $f,g$ are polynomials in $\q[x]$ with degrees $d$ and $e$ respectively. If $f\circ g$ is a $p^r$-pure polynomial, where $d>r$, and $g$ is $p$-type, then $f\left(x+g(0)\right)$ is $p^r$-pure.
\end{proposition}
\begin{proof}
 In order to use Lemma~\ref{lemma: If f composed with g is p-type then f(x+g(0)) and g(x)-g(0) is p-type}, we shall show that $\nu_p(f)=0$. Assume that $f(x)=ax^d+c_{d-1}x^{d-1}+\ldots+c_ix^i+\ldots+c_0$ is such that $\nu_p(c_i)<0$ with $i$ being the largest such integer, $0\le i\le d-1$. We write $g(x)=b_ex^e+p^sG(x)$ for some $G\in \q[x]$ with $\nu_p(G)=0$, $\nu_p(b_e)=0$ and $s\ge1$. One sees that
\begin{equation*}
f(g(x))=a{(b_ex^e+p^sG(x))}^d+\ldots+c_i{(b_ex^e+p^sG(x))}^i+\ldots+ c_0.    
\end{equation*}
The coefficient $c_ib^{i}_e$ has negative $p$-adic valuation, yet there is no monomial of any expansion of $f(g(x))$ that has degree $ei$ and a coefficient of negative $p$-adic valuation. Thus, such $c_i$ does not exist and $\nu_p(f)=0$.  Let $h(x)=f(x+g(0))$. By Lemma \ref{lemma: If f composed with g is p-type then f(x+g(0)) and g(x)-g(0) is p-type}, $h$ is $p$-type. Moreover, $\nu_p(h(0))=\nu_p(f(g(0))=r$, see Lemma \ref{lem: f(c)_where_f_is_pure}. So we are left with showing that $h$ satisfies iii) in Definition \ref{def: pure polynomials}. 

Suppose $h(x)=a_dx^d+\ldots+a_0$ and $g(x)-g(0)=b_ex^e+\ldots+ b_1x$. We assume that $k$, $0<k<d$, is the maximum positive integer such that $a_k$ does not satisfy iii), i.e.,  $\frac{\nu_p(a_k)}{d-k}< \frac{r}{d}$. Given that
\begin{equation*}
f(g(x))=h\left(g(x)-g(0)\right)=a_d{\left(b_ex^e+\ldots+b_1x\right)}^d+\ldots + a_k{\left(b_ex^e+\ldots+b_1x\right)}^k+\ldots+a_0,
\end{equation*}
the monomial $a_kb_e^kx^{ek}$ doesn't satisfy condition iii) as $\frac{\nu_p(a_k)}{de-ek}< \frac{r}{de}$. Yet, when added with monomials of the same degree, the sum should satisfy iii) as $f\circ g$ is $p^r$-pure. Thus, there has to be other monomials in the expansion of $h(g(x)-g(0))$ of degree $ek$ whose coefficients have $p$-adic valuation less than $\frac{(d-k)r}{d}$. For $j$, $0<j<k$, the monomials in the expansion of $a_j{(b_ex^e+\ldots+b_1x)}^j$ have degrees strictly less than $ek$. If $k<j\leq d$, then $\frac{\nu_p(a_j)}{d-j}\geq \frac{r}{d}$ and by Lemma~\ref{lem: purity_lemma}, $\frac{\nu_p(c)}{t-de}\geq \frac{r}{de}$ for any monomial $cx^t$ in the expansion of $a_j{(b_ex^e+\ldots+b_1x)}^j$. Therefore, such $k$ does not exist and $f(x+g(0))$ is $p^r$-pure.
\end{proof}
Before introducing the main theorem of this section, we first prove the following lemma.
\begin{lemma}\label{cor: negative_valuation}
If $f$ is a $p^r$-pure polynomial and $c\in \q$ is such that $\nu_p(c)\le 0$, then  $\nu_p\big(f(c)\big)\le 0$. 
\end{lemma}
\begin{proof}
Let $f(x)=\sum_{i=1}^da_ix^i$ be a $p^r$-pure polynomial. We claim that $ \displaystyle \min_{0\leq i\leq d-1}\{\nu_p(a_i)+i\nu_p(c)\}>d\nu_p(c)$, hence $\nu_p\big(f(c)\big)=d\nu_p(c)\le 0$. In order to prove this claim, we note that for any $i$, $0\leq i<d$, $\nu_p(a_i)+i\nu_p(c)\geq \frac{r}{d}(d-i)+i\nu_p(c)=r-\frac{ri}{d}+i\nu_p(c)>r-\frac{ri}{d}+d\nu_p(c)>d\nu_p(c)$, hence the result.   
\end{proof}
\begin{theorem}\label{thm: characterization_of_eventually_pure}
Let $r$ be a positive integer and $p$ be a prime. Suppose $f(x)=a_dx^d+\ldots+a_0 \in \q[x]$ is not $p^r$-pure and $d>r$. Then $f(x)$ is eventually $p^r$-pure if and only if the following conditions hold
\begin{itemize}
    \item[i)] $d=p^m$ for some $m\geq 1$,
    \item[ii)] $f(x)\equiv a_dx^d+a_0 \pmod p$ such that $\nu_p(a_d)=\nu_p(a_0)=0$,
    \item[iii)] $f(x+c)$ is $p^r$-pure for some $c\in \q$.
\end{itemize}
Moreover, the least integer $n>1$ such that $f^n$ is $p^r$-pure is given by
$n=p$ if $a_d\equiv 1 \pmod p$; or $n=\operatorname{ord}_p(a_d)$ otherwise.
\end{theorem}
\begin{proof}
Since $f$ is not $p$-type but eventually $p$-type, it follows that $d=p^m$ for some $m\geq 1$ and $f(x)\equiv a_dx^d+a_0 \pmod p$ such that $\nu_p(a_d)=\nu_p(a_0)=0$, see Theorem \ref{thm: characterization_of_eventually_p-type_polynomials}. Since $f^n(x)=f\left(f^{n-1}(x)\right)$ is $p^r$-pure, therefore by Proposition \ref{fg is pure then f(x+c)is pure}, one has $f\left(x+f^{n-1}(0)\right)$ is $p^r$-pure. 

Conversely, suppose $f$ satisfies conditions i), ii) and iii) in the statement of the theorem. Set $g(x)=f(x+c)$. Since $g$ is $p^r$-pure, it follows that $g$ is $p$-type. Theorem \ref{thm: characterization_of_eventually_p-type_polynomials} together with ii) imply the existence of an $n>1$ such that $f^n$ is $p$-type.  Note that $f^{n-1}(x)\equiv ax^{d^{n-1}}+f^{n-1}(0) \pmod p$ where $\nu_p\left(f^{n-1}(0)\right)=0$, see Lemma~\ref{lemma: iteration of a p-type mod p}. Given that $\nu_p\left(g\left(f^{n-1}(0)-c\right)\right)=\nu_p\left(f^{n}(0)\right)= r>0$  and that $g$ is $p^r$-pure, we must have $\nu_p(f^{n-1}(0)-c)>0$, see Lemma~\ref{cor: negative_valuation}. This implies that $f^{n-1}(x)-c$ is a $p$-type polynomial. Using Corollary~\ref{cor: Composition of pure with p-type when when r/d<1}, $f^n(x)=g\left(f^{n-1}(x)-c\right)$ is $p^r$-pure.
Finally, the value of $n$ is given by Proposition~\ref{prop: which iteration of an eventually p-type is p-type}.
\end{proof}
The previous Theorem gives rise to the following family of dynamically irreducible polynomials.
\begin{corollary}\label{cor: Characterization_of_eventually_Dumas}
Let $r\ge 1$ be an integer and $p$ be a rational prime. 
Let $f(x)=a_dx^d+\ldots+a_0 \in \q[x]$, $d>r$, be such that $f$ is not $p^r$-Dumas.

There is an integer $n\ge 2$ such that $f^n$ is $p^r$-Dumas if and only if the following conditions hold
\begin{itemize}
    \item[i)] $d=p^m$, for some $m\ge 1$,
    \item[ii)] $f(x)\equiv a_dx^{d}+a_0 \pmod p$ with $\nu_p(a_d)=\nu_p(a_0)=0$,
    \item[iii)] $f(x+c)$ is $p^r$-Dumas for some $c\in \q$.
\end{itemize}
\end{corollary}
\begin{example}
Consider the family of polynomials
\begin{equation*}
   f(x)={(x+a)}^{p^m}+b \in \q[x] \text{; }\quad \nu_p(a)\geq 0\text{ and } 1\leq \nu_p(b) < p^m. 
\end{equation*}
In view of Theorem \ref{thm: characterization_of_eventually_pure}, $f(x)$ is eventually $p^r$-pure where $r=\nu_p(b)$. In addition, $f(x)$ is eventually stable over $\q$, see Corollary \ref{cor1}. In fact, if $\gcd\left(p^m,\nu_p(b)\right)=1$, then $f$ is eventually $p^r$-Dumas by Corollary \ref{cor: Characterization_of_eventually_Dumas}, hence $f$ is dynamically irreducible; otherwise, the number of irreducible factors of any iterate of $f$ is at most $\max\{\gcd\left(p^{nm},\nu_p(b)\right): n\geq 1\}$.  
\end{example}
In light of Theorem \ref{thm: characterization_of_eventually_pure} and Corollary \ref{cor: Characterization_of_eventually_Dumas}, it is reasonable to ask the following question.
\begin{question}\label{question: z-shifts_of_eventually_stable}
 If $f\in \z[x]$ is eventually stable (respectively, dynamically irreducible) over $\q$, is $f(x+c)$ eventually stable (respectively, dynamically irreducible) for any $c\in \z$?
\end{question}
The following examples provide a negative answer to the latter question. 
\begin{example}
The polynomial 
\begin{equation*}
f(x)=x^2 + 5x + 5   
\end{equation*} 
is $5$-Eisenstein, hence dynamically irreducible over $\q$. However,  
 $g(x)=f(x-3)=x^2-x-1$ is $3$-newly reducible. More precisely, $g^2(x)$ is irreducible, but $g^3(x)=(x^4 - 3 x^3 + 4 x - 1) (x^4 - x^3 - 3 x^2 + x + 1) $.
\end{example}
\begin{example}\label{ex: shift_eventually_stable}
The polynomial
\begin{equation*}
 f(x)=x^2+8x+12   
\end{equation*}
 is $2^2$-pure, hence eventually stable by Corollary \ref{pure polynomials are eventually stable}. However,  
\begin{equation*}
 f(x-3)=x^2+2x-3   
\end{equation*}
is not eventually stable as it belongs to the family $f_k(x)=x^2+kx-(k+1)\in \z[x]$ which is not eventually stable because $0$ is periodic under $f$, see \cite{jones_2008}.
\end{example}
Now, as an application of Theorem~\ref{thm: characterization_of_eventually_pure}.  we present a family of polynomials that answers Question \ref{question: z-shifts_of_eventually_stable} positively.
\begin{corollary}\label{cor: z-shifts_of_pure_polynomials}
Let $f\in \z[x]$ be a $p^r$-pure (respectively, $p^r$-Dumas) polynomial of degree $p^m>r$. The polynomial $f(x+c)$ is eventually stable (respectively, dynamically irreducible) for all $c\in \z$. In general, if $f\in \q[x]$ is a $p^r$-pure (respectively, $p^r$-Dumas) polynomial of degree $p^m>r$, then $f(x+c)$ is eventually stable (respectively, dynamically irreducible) for all $c \in \q$ with $\nu_p(c)\geq 0$.
\end{corollary}
\begin{proof}
Let $f \in \q[x]$ be a $p^r$-pure (respectively, $p^r$-Dumas) of degree $p^m>r$ and $c\in \q$. If $\nu_p(c)>0$, then $x+c$ is a $p$-type polynomial and $f(x+c)$ is $p^r$-pure (respectively, $p^r$-Dumas), see Theorem~\ref{thm: Composition_of_pure_and_p-type} (respectively, Corollary \ref{cor: composition of dumas and p-type}). If $\nu_p(c)=0$, then $g(x)=f(x+c)$ satisfies conditions i), ii) and iii) in Theorem~\ref{thm: characterization_of_eventually_pure} (respectively, Corollary~\ref{cor: Characterization_of_eventually_Dumas}) and thus $g$ is eventually $p^r$-pure (respectively, $p^r$-Dumas).   
\end{proof}

 Note that we only dealt with eventually $p^r$-pure polynomials with degree $d>r$. This suggests the following question. 
\begin{question}
If $f$ is a $p^r$-pure polynomial with degree $p^m\leq r$, is $f(x+c)$ eventually $p^r$-pure for any rational $c$ with $\nu_p(c)\geq 0$?
\end{question}
 In fact, the polynomial $f(x)$ in Example \ref{ex: shift_eventually_stable} is a $2^2$-pure polynomial with $\deg f=2<2^2$, yet $f(x-3)$ is not eventually pure as it is not eventually stable.  
 
Next, the following result is an application of Corollary \ref{cor: dumas polynomials are dynamically irreducible} and  Corollary \ref{cor: Characterization_of_eventually_Dumas}. 
\begin{corollary}\label{brown application}
Let $p$ be a prime. Suppose $f,g$ are monic polynomials in $\q[x]$ such that $\overline{g}$ is the reduction of $g$ modulo $p$ and $\operatorname{deg}(g)=\operatorname{deg}(\overline{g})=e$.
If $f$ is eventually $p^r$-Dumas for some iteration $n\geq 1$, $\overline{g}$ is irreducible in $\mathbb{F}_p[x]$ and $\operatorname{gcd}\left(e,r\right)=1$, then $f^{kn}\circ g$ is irreducible in $\q[x]$ for all $k\geq 1$. In addition, if $\overline{g}$ is dynamically irreducible in $\mathbb{F}_p[x]$, then $f^{kn}\circ g^m$ is irreducible in $\q[x]$ for all $k,m\geq 1$ and $f^{kn}$ is $g$-stable for all $k\geq 1$ in $\q[x]$.
\end{corollary}

The proof of the previous Corollary depends on a special case of the generalized Sch{\"o}nemann polynomial discussed in \cite{bishnoi}. We present this special case as a lemma.
\begin{lemma}\label{special case of brown}
 Let $A$ and $g$ be polynomials in $\q[x]$. Assume that the $g$-expansion of the polynomial $A$ in $\q[x]$ is given by
 \begin{equation*}
    A=a_d{g}^d+\ldots+a_1g+a_0.
 \end{equation*}
for some $a_0,\ldots,a_d\in \q[x]$. Suppose there exists a prime $p$ such that:
\begin{enumerate}[label=(\alph*)]
    \item The reduction of $g$ modulo $p$ is irreducible over $\mathbb{F}_p$,
    \item $a_d(x)=1$,
    \item $\frac{\nu_p(a_i)}{d-i} \geq \frac{\nu_p(a_0)}{d}>0$ for all $1 \leq i\leq d-1$,
    \item $\operatorname{gcd}\left(\nu_p(a_0),d\right)=1$.
\end{enumerate}
Then $A$ is irreducible in $\q[x]$. 
\end{lemma}
 We remark that if we force $a_0,\ldots,a_d$ to be constant polynomials and $g(x)=x$, we deduce the monic case of Definition \ref{def: dumas criterion}.
\begin{proof}[Proof of Corollary \ref{brown application}]
Assume that  $f^n(x)=x^d+\ldots+a_0$ is $p^r$-Dumas for some iterate $n\geq 1$. We write 
\begin{equation*}
 A(x)=f^n\left(g(x)\right)= g(x)^d+a_{d-1}g(x)^{d-1}+\ldots+a_1g(x)+a_0.   
\end{equation*}
By assumption, the polynomial $f^k\circ g$ satisfies the conditions in Lemma \ref{special case of brown}, hence $A(x)$ is irreducible in $\q[x]$. If $k=1$ (respectively, $k>1$), then by Corollary \ref{cor: dumas polynomials are dynamically irreducible} (respectively, Corollary \ref{cor: Characterization_of_eventually_Dumas}), the polynomials $f^n$ (respectively, $f^{kn}$) are $p^r$-Dumas for all $n\geq 1$. It follows by Lemma~\ref{special case of brown}, $f^n\circ g$ (respectively, $f^{kn}\circ g$) is irreducible for all $n\geq 1$. Finally, if $g$ is dynamically irreducible over $\mathbb{F}_p$, then, $f^n\circ g^m$ (respectively, $f^{kn}\circ g^m$) is irreducible for any $n,m\geq 1$.    
\end{proof}
The following is another result that combines irreducibility over finite fields and irreducibility over number fields.
\begin{corollary}\label{finite, number and iteration}
Let $f$ be a $p^r$-Dumas polynomial, $r\geq 1$, and let $\alpha$ be a root of $f$.
Let $g\in \q[x]$ be such that the reduction $\overline{g}$ of $g$ modulo $p$ is irreducible over $\mathbb{F}_p$. Then $g(x)-\alpha$ is irreducible over the number field $\q(\alpha)$. 
\end{corollary}
\begin{proof}
 By Lemma \ref{special case of brown}, $f\circ g$ is irreducible over $\q$. By Capelli's Lemma, \cite[Lemma 1]{Ayad2000}, if $\alpha$ is a root of $f$, then the polynomial $g(x)-\alpha$ must be irreducible over $\q(\alpha)$.  
\end{proof}
We end this section with the following example.
\begin{example}
Consider the polynomial
\begin{equation*}
    g(x)=x^2+1\in \q[x].
\end{equation*}
We consider the polynomial $\overline{g}$ in $\mathbb{F}_3[x]$. Since $-\overline{g}(\gamma)=2=\overline{g}^n(\gamma)$ for all $n\geq 2$, and $2$ is a nonsquare in $\mathbb{F}_3$, it follows that $\overline{g}$ is dynamically irreducible over $\mathbb{F}_3$, see Proposition \ref{quadratic stable}. According to Corollary \ref{brown application}, if $f\in \q[x]$ is a monic  $3^r$-Dumas polynomial for some $r\geq 1$, then $f^n\circ g^m \in \q[x]$ is irreducible for all $n,m\geq 1$. Moreover, if $f$ is eventually $3^r$-Dumas, then by Theorem \ref{thm: characterization_of_eventually_pure}, the iterates $f^{kp}$ are $3^r$-Dumas for all $k\geq 1$. It follows again by Corollary \ref{brown application} that $f^{kp}\circ g^m$ is irreducible over $\q$.    
\end{example}

\bibliographystyle{alpha}
\bibliography{ref}

\end{document}